\documentclass[11pt, a4paper, twoside,leqno]{amsart}
\usepackage[centering, totalwidth = 420pt, totalheight = 610pt]{geometry}
\usepackage{amssymb, amsmath, amsthm, enumerate, microtype, stmaryrd, url,mathrsfs}
\usepackage[latin1]{inputenc}
\usepackage[dvips, arrow, matrix, tips, curve]{xy}
\usepackage{color}
\definecolor{darkgreen}{rgb}{0,0.45,0}
\usepackage[pagebackref,colorlinks,citecolor=darkgreen,linkcolor=darkgreen]{hyperref}
\SelectTips{cm}{10}

\newcommand{\cat}[1]{\mathbf{#1}}
\newcommand{\op}{\mathrm{op}}
\newcommand{\id}{\mathrm{id}}
\newcommand{\thg}{{\mathord{\text{--}}}}
\newcommand{\set}[2]{\left\{\,#1 \ \vrule\  #2\,\right\}}
\newcommand{\spn}[1]{{\left<{#1}\right>}}
\newcommand{\cd}[2][]{\vcenter{\hbox{\xymatrix#1{#2}}}}
\newcommand{\ccd}[2][]{\hbox{\xymatrix#1{#2}}}

\newcommand{\A}{{\mathcal A}}

\newcommand{\C}{{\mathcal C}}
\newcommand{\D}{{\mathcal D}}

\newcommand{\I}{{\mathcal I}}

\newcommand{\K}{{\mathcal K}}

\newcommand{\ELL}{{\mathcal L}}

\newcommand{\N}{{\mathcal N}}

\newcommand{\OH}{{\mathcal O}}

\newcommand{\R}{{\mathcal R}}
\newcommand{\T}{{\mathcal T}}

\newcommand{\pullbackcorner}[1][dr]{\save*!/#1+1.2pc/#1:(1,-1)@^{|-}\restore}
\newcommand{\pushoutcorner}[1][dr]{\save*!/#1-1.2pc/#1:(-1,1)@^{|-}\restore}

\newtheorem{Thm}{Theorem}
\newtheorem{Prop}[Thm]{Proposition}

\theoremstyle{definition}

\newtheorem{Ex}[Thm]{Example}

\newtheorem{Rk}[Thm]{Remark}

\newcommand{\wo}{\mathrel \pitchfork}

\begin{document}

\title[A homotopy-theoretic universal property\dots]{A homotopy-theoretic universal property\\of Leinster's operad for weak $\omega$-categories}
\author
       {Richard Garner}
\address{       Department of Pure Mathematics and Mathematical Statistics, Wilberforce Road,
       Cambridge CB3 0WB, UK}\email{rhgg2@cam.ac.uk}
     
%\amsclass{18D50, 55U35, 18D05}
\begin{abstract}
We explain how any cofibrantly generated weak factorisation system on a
category may be equipped with a universally and canonically determined choice
of cofibrant replacement. We then apply this to the theory of weak
$\omega$-categories, showing that the universal and canonical cofibrant
replacement of the operad for strict $\omega$-categories is precisely
Leinster's operad for weak $\omega$-categories.
\end{abstract}
\thanks{The author acknowledges the support of a Marie Curie Intra-European
Fellowship, Project No.\ 040802, and a Research Fellowship of St John's
College, Cambridge.}
 \maketitle

\section{Introduction}
One of the most interesting aspects of modern homotopy theory is the general
machinery it provides for replacing some piece of algebraic structure with a
``weakened'' version of that same structure. The picture is as follows: we
begin with a category~$\C$ equipped with a notion of higher-dimensionality
coming from a model structure in the sense of
Quillen~\cite{Quillen1967Homotopical}. We now contemplate some notion of
algebraic theory on~$\C$: monads, operads, or Lawvere theories on~$\C$, for
example. These algebraic theories themselves form a category~$\cat{Th}(\C)$,
and by making use of various transfer techniques we obtain a model structure
on~$\cat{Th}(\C)$ from the one on~$\C$. Now, for a particular algebraic theory
$T \in \cat{Th}(\C)$, we obtain a weakened version of this theory by taking a
cofibrant replacement for~$T$ in the category~$\cat{Th}(\C)$. A cofibrant
replacement is a generalised projective resolution: and so the effect this has
is to transform each algebraic law satisfied by the theory~$T$ into a piece of
higher-dimensional data witnessing the weak satisfaction of that same law, with
all this extra data fitting together in a coherent way.

The remarkable thing about this machinery is how little it requires to get
going. All we need is a category~$\C$, a notion of algebraic structure, and a
notion of higher-dimensionality; and for this last, we do not even need a full
model structure on~$\C$. A single weak factorisation
system~\cite{Bousfield1977Constructions} will do, and for a sufficiently well
behaved (typically, locally presentable)~$\C$ we may obtain this by using the
small object argument of Quillen~\cite{Quillen1967Homotopical} and
Bousfield~\cite{Bousfield1977Constructions}: for which it suffices to specify a
set of generating higher-dimensional cells in~$\C$, together with their
boundaries and the inclusions of the latter into the former. Moreover, it is
frequently the case that $\C = [\D^\op, \cat{Set}]$ for some Reedy
category~$\D$~\cite{Hirschhorn2003Model, Reedy1974Homotopy}, in which case we
have canonical notions of both \emph{cell} (the representable presheaves) and
\emph{boundary} (arising from the Reedy structure).

Yet this rather appealing construction has a problem, which arises when we ask
what ``the'' weakened version of a particular algebraic theory~$T$ is. Because
cofibrant replacements need not be unique, even up to isomorphism, we may only
legitimately talk of ``a''~weakened version of~$T$; and so it becomes pertinent
to ask which one we choose. The usual answer given is that we don't really
care, since all the choices are essentially equivalent: but since the point of
being algebraic is in some sense to ``pin down everything that can be pinned
down'', it seems perverse that we should be so hazy on this particular point.

The obvious solution is to make a definite choice of cofibrant replacements
in~$\cat{Th}(\C)$: and if~--~as is almost always the case~--~the weak
factorisation system under consideration was constructed using the small object
argument, then it may be equipped with such a choice. Yet the situation is not
entirely satisfactory for two reasons. Firstly, the cofibrant replacements we
obtain are in no way canonical, since the induction which constructs them is
governed by some (sufficiently large) regular cardinal~$\alpha$, with different
choices of~$\alpha$ leading to different cofibrant replacements. Secondly and
more importantly, the cofibrant replacements we obtain are neither particularly
natural nor computationally tractable. In principle, it would be possible to
reason about them by induction; but in practice, this would require some rather
strange combinatorics of a nature entirely orthogonal to that of the
mathematics one was trying to do.

However, recent work of Grandis \& Tholen~\cite{Grandis2006Natural} and the
author~\cite{Garner2008Understanding} suggests a solution to this problem.
Using the results of~\cite{Garner2008Understanding}, we may equip any
reasonable (which is to say, cofibrantly generated) weak factorisation system
with a canonical and universal notion of cofibrant replacement. The canonicity
says that we need specify no additional information beyond the set of
generating cells and boundaries; whilst the universality tells us that the
cofibrant replacements we obtain are rather natural, and in particular admit a
straightforward calculus of inductive reasoning.

In this note, we first explain the technology behind these universal
co\-fibrant replacements, and then illustrate their utility by means of an
example drawn from the study of weak \mbox{$\omega$-categories}. More
specifically, we consider Batanin's theory of globular
operads~\cite{Batanin1998Monoidal}, and by using the machinery outlined above,
obtain a canonical and universal notion of cofibrant replacement for globular
operads. We then show that applying this to the globular operad for strict
$\omega$-categories yields precisely the operad singled out by
Leinster~\cite{Leinster2004Operads} as the operad  for weak
$\omega$-categories.

\section{Weak factorisation and cofibrant replacement}
\subsection{Weak factorisation systems}
A \emph{weak factorisation system} \cite{Bousfield1977Constructions} $(\ELL,
\R)$ on a category~$\C$ is given by two classes $\ELL$ and~$\R$ of morphisms
in~$\C$ which are each closed under retracts when viewed as full subcategories
of the arrow category $\C^\mathbf 2$, and which satisfy the two axioms of
\begin{enumerate}
\item[(i)] \emph{factorisation}: each $f \in \C$ may be written as $f = pi$
    where $i \in \ELL$ and $p \in \R$; and
\item[(ii)] \emph{weak orthogonality}: for each $i \in \ELL$ and $p \in
    \R$, we have $i \wo p$,
\end{enumerate}
where to say that $i \wo p$ holds is to say that for each commutative square
\begin{equation}\tag{$\star$} \label{cs}
    \cd{
      U \ar[r]^{f} \ar[d]_{i} &
      W \ar[d]^{p} \\
      V \ar[r]_{g}  &
      X
    }
\end{equation}
we may find a filler $j \colon V \to W$ satisfying $ji = f$ and $pj = g$. For
those weak factorisation systems that we will be considering, the following
terminology will be appropriate: the maps in $\ELL$ we call
\emph{cofibrations}, and the maps in $\R$, \emph{acyclic fibrations}. Supposing
$\C$ to have an initial object $0$, we say that $U \in \C$ is \emph{cofibrant}
just when the unique map $0 \to U$ is a cofibration; and define a
\emph{cofibrant replacement} for $X \in \C$ to be a factorisation of the unique
map $0 \to X$ as a cofibration followed by an acyclic fibration:
\begin{equation*}
    0 \xrightarrow {} X' \xrightarrow p X\text.
\end{equation*}
The principal tool we use for the construction of weak factorisation systems is
the following result, first proved by Quillen in the finitary case
\cite[Chapter II, \S 3]{Quillen1967Homotopical} and in the following
transfinite form by Bousfield \cite{Bousfield1977Constructions}. For a modern
account, see \cite{Hovey1999Model}, for example.

\begin{Prop}[The small object argument]\label{smallobj} Let $\C$ be a locally
presentable category, and let $I$ be a set of maps in $\C$. Define classes of
maps $I^\downarrow$ and $I^{\downarrow\uparrow}$ by
\begin{equation*}
    I^{\downarrow} := \set{p \in \C^\mathbf 2}{j \wo p \text{ for all $j \in I$}}
\end{equation*}
and
\begin{equation*}
    I^{\downarrow \uparrow} := \set{i \in \C^\mathbf 2}{i \wo p \text{ for all $p \in \smash{I^\uparrow}$}}\text.
\end{equation*}
Then the pair $(I^{\downarrow \uparrow}, I^{\downarrow})$ is a weak
factorisation system on $\C$.
\end{Prop}
We call $I$ the set of \emph{generating cofibrations} for $(I^{\downarrow
\uparrow}, I^\downarrow)$, and given a map $i \colon U \to V$ in $I$, we call
$V$ a \emph{generating cell} and $U$ its \emph{boundary}. To say that a weak
factorisation system $(\ELL, \R)$ is \emph{cofibrantly generated} is to say
that there is some set $I$ for which $(\ELL, \R) = (I^{\downarrow \uparrow},
I^\downarrow)$. Observe that this $I$ will usually not be unique; however, this
is not a problem since we typically begin with the set $I$ and generate the
weak factorisation system from it, rather than vice versa.

\subsection{Functorial w.f.s.'s} Given a w.f.s.\ $(\ELL, \R)$ on a category $\C$, it may be the case that
for each morphism $f \colon X \to Y$ of $\C$, we can provide a choice
\begin{equation*}
    X \xrightarrow{\lambda_f} Kf \xrightarrow{\rho_f} Y
\end{equation*}
of $(\ELL, \R)$ factorisation for $f$. Suppose this is so; then by weak
orthogonality, we know that for each square as on the left of the following
diagram, there exists a filler for the corresponding square on the right:
\begin{equation*}
    \cd{
      U \ar[r]^{h} \ar[d]_{f} &
      W \ar[d]^{g} \\
      V \ar[r]_{k} &
      X
    } \qquad \dashrightarrow \qquad
    \cd{
      U \ar[r]^{\lambda_g. h} \ar[d]_{\lambda_f} &
      Kg \ar[d]^{\rho_g} \\
      Kf \ar[r]_{k . \rho_f} \ar@{.>}[ur] &
      X\text.
    }
\end{equation*}
It may now be that we can choose a diagonal filler $K(h, k) \colon Kf \to Kg$
for each such square: and that we can do so in such a way that the assignations
$f \mapsto Kf$ and \mbox{$(h, k) \mapsto K(h, k)$} underlie a functor $K \colon
\C^\mathbf 2 \to \C$. If this is so, then the maps $\lambda_f$ and $\rho_f$
necessarily provide the components of natural transformations $\lambda \colon
\mathrm{cod} \Rightarrow K$ and $\rho \colon K \Rightarrow \mathrm{dom}$; and
we call the triple $(K, \lambda, \rho)$ so obtained a \emph{functorial
realisation} of $(\ELL, \R)$.

\begin{Prop}\label{realisation}Let $\C$ be a locally
presentable category, and let $I$ be a set of maps in $\C$. Then for each
choice of a sufficiently large regular cardinal $\alpha$, the small object
argument provides us with a functorial realisation $(K^{(\alpha)},
\lambda^{(\alpha)}, \rho^{(\alpha)})$ of the weak factorisation system
$(I^{\downarrow\uparrow}, I^\downarrow)$.
\end{Prop}
\noindent The proof falls out of the construction used in the small object
argument. The regular cardinal $\alpha$ that we provide serves to fix the
length of the transfinite induction by which factorisations are constructed.
Note that the functorial realisation we obtain depends not only upon $\alpha$
but also upon the particular set $I$ of generating cofibrations that we choose.

\begin{Rk} \label{rk1}It was shown in \cite[\S 2.4]{Rosick'y2002Lax} that the data $(K, \lambda, \rho)$ for a functorial
realisation completely determines the underlying w.f.s.\ $(\ELL, \R)$. To see
this, we define two auxiliary functors $L, R \colon \C^\mathbf 2 \to \C^\mathbf
2$ whose action on objects is given by
\begin{equation*}
    L\left(\cd{X \ar[d]^f \\ Y}\right) = \cd{X \ar[d]^{\lambda_f} \\ Kf}
\qquad \text{and} \qquad
    R\left(\cd{X \ar[d]^f \\ Y}\right) = \cd{Kf \ar[d]^{\rho_f} \\ Y\text;}
\end{equation*}
and two auxiliary natural transformations $\Lambda \colon \id_{\C^\mathbf 2}
\Rightarrow R$ and $\Phi \colon L \Rightarrow \id_{\C^\mathbf 2}$ whose
respective components at $f \colon X \to Y$ are:
\begin{equation*}
    \Lambda_f = \cd{X \ar[d]_f \ar[r]^{\lambda_f} & Kf \ar[d]^{\rho_f} \\ Y \ar[r]_{\id_Y} & Y}
\qquad \text{and} \qquad
    \Phi_f = \cd{X \ar[d]_{\lambda_f} \ar[r]^{\id_X} & X \ar[d]^{f} \\
Kf \ar[r]_{\rho_f} & Y\text.}
\end{equation*}
We may now show that a morphism $f \colon X \to Y$ of $\C$ lies in $\R$ just
when the map $\Lambda_f \colon f \to Rf$ admits a retraction in $\C^\mathbf 2$:
which is to say that $f$ is an algebra for the pointed endofunctor $(R,
\Lambda)$. Dually, we may show that $f$ lies in $\ELL$ just when the map
$\Phi_f \colon Lf \to f$ admits a section: which is to say that $f$ is a
coalgebra for the copointed endofunctor $(L, \Phi)$. As a particular case of
this last fact, if $\C$ has an initial object, then $L \colon \C^\mathbf 2 \to
\C^\mathbf 2$ restricts and corestricts to the coslice $0 / \C \cong \C$ to
yield a \emph{cofibrant replacement functor} $Q \colon \C \to \C$ together with
a copointing $\epsilon \colon Q \Rightarrow \id_\C$; and now an an object $X
\in \C$ is cofibrant just when it may be made into a coalgebra for $(Q,
\epsilon)$; which is to say, just when $\epsilon_X \colon QX \to X$ admits a
retraction in $\C$.
\end{Rk}

\subsection{Natural w.f.s.'s} As we mentioned in the Introduction, the functorial realisations
$(K^{(\alpha)}, \lambda^{(\alpha)}, \rho^{(\alpha)})$ that we obtain from the
small object argument are not very intuitive. One way of rectifying this is
through a further strengthening of the notion of weak factorisation system. We
begin from a w.f.s.\ $(\ELL, \R)$ on a category $\C$ together with a functorial
realisation $(K, \lambda, \rho)$ thereof. Now, since in any w.f.s.\ the classes
of $\ELL$-maps and $\R$-maps are closed under composition, we have fillers for
squares of the following form:
\begin{equation*}
    \cd{
      X \ar[r]^{\lambda_{\lambda_f}} \ar[d]_{\lambda_f} &
      K\lambda_f \ar[d]^{\rho_f.\rho_{\lambda_f}} \\
      Kf \ar[r]_{\rho_f} \ar@{.>}[ur] &
      Y
    } \qquad \text{and} \qquad
    \cd{
      X \ar[r]^{\lambda_f} \ar[d]_{\lambda_{\rho_f}.\lambda_f} &
      Kf \ar[d]^{\rho_f} \\
      K\rho_f \ar[r]_{\rho_{\rho_f}} \ar@{.>}[ur] &
      Y\text,
    }
\end{equation*}
and it may be the case that we can provide a choice of fillers $\sigma_f \colon
Kf \to K\lambda_f$ and $\pi_f \colon K\rho_f \to Kf$ for each such square; and
that we can do so in such a way that the morphisms $\Sigma_f \colon Lf \to LLf$
and $\Pi_f \colon RRf \to Rf$ of $\C^\mathbf 2$ given by
\begin{equation*}
    \Sigma_f =\cd{X \ar[d]_{\lambda_f} \ar[r]^{\id_X} & X \ar[d]^{\lambda_{\lambda_f}} \\
Kf \ar[r]_{\sigma_f} & K\lambda_f}\qquad \text{and} \qquad
    \Pi_f =  \cd{K\rho_f \ar[d]_{\rho_{\rho_f}} \ar[r]^{\pi_f} & Kf \ar[d]^{\rho_f} \\ Y
\ar[r]_{\id_Y} & Y}
\end{equation*}
provide the components at $f$ of natural transformations $\Sigma \colon L
\Rightarrow LL$ and \mbox{$\Pi \colon RR \Rightarrow R$}. Under these
circumstances, it may be that \mbox{$\mathsf R = (R, \Lambda, \Pi)$} describes
a monad on $\C^\mathbf 2$, and $\mathsf L = (L, \Phi, \Sigma)$ a comonad; and
that the natural transformation $\Delta \colon LR \Rightarrow RL \colon
\C^\mathbf 2 \to \C^\mathbf 2$ with components
\begin{equation*}
    \Delta_f =\cd{Kf \ar[d]_{\lambda_{\rho_f}} \ar[r]^{\sigma_f} & X \ar[d]^{\rho_{\lambda_f}} \\
K\rho_f \ar[r]_{\pi_f} & Kf}
\end{equation*}
describes a distributive law~\cite{Beck1969Distributive} between $\mathsf L$
and $\mathsf R$. Under these circumstances, we will say that $(\mathsf L,
\mathsf R)$ is an \emph{algebraic realisation} of $(\ELL, \R)$. The pairs
$(\mathsf L, \mathsf R)$ arising in this way are the \emph{natural weak
factorisation systems} of \cite{Grandis2006Natural}. Though the requirements
for an algebraic realisation may appear strong, they are in fact rather easily
satisfied:

\begin{Prop}[The refined small object argument] \label{refined}Let $\C$ be a locally
presentable category, and let $I$ be a set of maps in $\C$. Then the weak
factorisation system $(I^{\downarrow\uparrow}, I^\downarrow)$ has an algebraic
realisation $(\mathsf L, \mathsf R)$.
\end{Prop}
\begin{proof}
For a full proof see \cite[Theorem 4.4]{Garner2008Understanding}: we recall
only the salient details here. We begin exactly as in the small object
argument. Given a map $f \colon X \to Y$ of $\C$ we consider the set
\begin{equation*}
    S := \set{(j, h, k)}{j : A \to B \in I\text,\, (h, k) \colon j \to f \in \C^\mathbf 2}
\end{equation*}
We have a commutative diagram
\begin{equation*}
    \cd[@C+2em]{
    \sum_{x \in S} A_x \ar[d]_{\sum_{x \in S} j_x} \ar[r]^-{\spn{h_x}_{x \in S}} &
    X \ar[d]^f \\
    \sum_{x \in S} B_x \ar[r]_-{\spn{h_x}_{x \in S}} &
    Y}
\end{equation*}
in $\C$; and may factorise this as
\begin{equation*}
    \cd[@C+2em]{
    \sum_{x \in S} A_x \ar[d]_{\sum_{x \in S} j_x} \ar[r]^-{\spn{h_x}_{x \in S}} &
    X \ar[d]^{\lambda'_f} \ar[r]^{\id_X} & X \ar[d]^f \\
    \sum_{x \in S} B_x \ar[r] &
    K'f \ar[r]_{\rho'_f} & Y}
\end{equation*}
where the left-hand square is a pushout. The assignation $f \mapsto \rho'_f$
may now be extended to a functor $R' \colon \C^\mathbf 2 \to \C^\mathbf 2$;
whereupon the map $(\lambda'_f, \id_Y) \colon f \to R'f$ provides the component
at $f$ of a natural transformation $\Lambda' \colon \id_{\C^\mathbf 2}
\Rightarrow R'$. We now obtain the monad part $\mathsf R$ of the desired
algebraic realisation as the free monad on the pointed endofunctor $(R',
\Lambda')$. We may construct this using the techniques
of~\cite{Kelly1980unified}. To obtain the comonad part $\mathsf L$ we proceed
as follows. The assignation $f \mapsto \lambda'_f$ underlies a functor
\mbox{$L' \colon \C^\mathbf 2 \to \C^\mathbf 2$}; and a little manipulation
shows that this functor in turn underlies a comonad $\mathsf{L}'$ on
$\C^\mathbf 2$. We may now adapt the free monad construction so that at the
same time as it produces $\mathsf R$ from $(R', \Lambda')$, it also produces
$\mathsf L$ from~$\mathsf L'$.
\end{proof}

\noindent The algebraic realisation of $(I^{\downarrow\uparrow}, I^\downarrow)$
given in Proposition \ref{refined} satisfies a universal property with respect
to $I$ which determines it up to unique isomorphism. Thus we refer to $(\mathsf
L, \mathsf R)$ as the \emph{universal algebraic realisation of~$I$}. To give
its universal property, we consider algebras for the monad $\mathsf R$. From
Remark~\ref{rk1}, we know that acyclic fibrations coincide with algebras for
the pointed endofunctor $(R, \Lambda)$: and so algebras for the monad $\mathsf
R$ must be  acyclic fibrations equipped with certain extra data. The following
Proposition makes this precise.
\begin{Prop}\label{characteriseralg}Let $(\mathsf L, \mathsf R)$ be the universal algebraic realisation of a
set of maps $I$ as given in Proposition \ref{refined}. To give an algebra for
the monad $\mathsf R$ on $\C^\mathbf 2$ is to give a map $p \colon W \to X$ of
$\C$ together with, for each $i \colon U \to V$ in $I$ and commutative square
\begin{equation*}
    \cd{
      U \ar[r]^{f} \ar[d]_{i} &
      W \ar[d]^{p} \\
      V \ar[r]_{g} &
      X
    }
\end{equation*}
a choice of diagonal filler $j \colon V \to W$, subject to no further
conditions, whilst to give a morphism of $\mathsf R$-algebras is to give a map
of $\C^\mathbf 2$ which strictly commutes with the chosen liftings. Moreover,
this characterisation of the category of $\mathsf R$-algebras determines the
pair $(\mathsf L, \mathsf R)$ up to unique isomorphism.
\end{Prop}
\begin{proof}
See \cite[Proposition 5.4]{Garner2008Understanding}.
\end{proof}
Dually, we may think of coalgebras for the comonad $\mathsf L$ as cofibrations
equipped with extra data. There is much less that can be said about these at a
general level: however, a good intuition is that if the cofibrations are
retracts of relatively free things then the $\mathsf L$-coalgebras are the
relatively free things of which they are retracts.\footnote{The problem of
ascertaining circumstances under which this intuition is valid is closely
related to the following old problem: given an adjunction which is known to be
monadic, under which circumstances is it also comonadic?}

\begin{Rk}\label{cofibcomonad}
Observe that for any algebraically realised w.f.s.\ $(\mathsf L, \mathsf R)$ on
a category with initial object, the chosen cofibrant replacements underlie a
\emph{cofibrant replacement comonad} \mbox{$\mathsf Q = (Q, \epsilon, \delta)$}
which is the restriction and corestriction of $\mathsf L$ to the coslice under
$0$. The concept of a cofibrant replacement comonad was first considered by
\cite{Radulescu-Banu1999Cofibrance}, though it should be noted that the
comonads constructed there do not coincide with the ones obtained from
Proposition~\ref{refined}. Indeed, they are built using the small object
argument, and so suffer from the same dependence upon a regular cardinal
$\alpha$ that we noted in Proposition~\ref{realisation}.
\end{Rk}

\begin{Ex}\label{ex2}
Consider the category $\cat{Ch}(R)$ of positively graded chain complexes of
$R$-modules, equipped with the set of generating cofibrations $I :=
\set{\partial y(i) \hookrightarrow y(i)}{i \in \mathbb N}$. Here $y(i)$ is the
representable chain complex at $i$, with components given by
\begin{equation*}
    y(i)_n = \begin{cases} R & \text{if $n = i$ or $n = i-1$;} \\ 0 &
    \text{otherwise,}\end{cases}
\end{equation*}
and as differential, the identity map $R \to R$ at stage $i$ and the zero map
elsewhere. The chain complex $\partial y(i)$ is its boundary, whose components
are
\begin{equation*}
    \partial y(i)_n = \begin{cases} R & \text{if $n = i-1$;} \\ 0 &
    \text{otherwise,}\end{cases}
\end{equation*}
and whose differential is everywhere zero. Since $\cat{Ch}(R)$ is locally
finitely presentable, we may apply Proposition~\ref{refined} to obtain an
algebraically realised w.f.s.\ $(\mathsf L, \mathsf R)$. We describe the
cofibrant replacement $\epsilon_X \colon QX \to X$ that this provides for $X
\in \cat{Ch}(R)$. The chain complex $QX$ will be free in every dimension; and
so it suffices to specify a set of free generators for each $(QX)_i$ and to
specify where each generator should be sent by the differential $d'_i \colon
(QX)_i \to (QX)_{i-1}$ and the counit $\epsilon_i \colon (QX)_i \to X_i$. We do
this by induction over $i$:
\begin{itemize}
\item For the base step, $(QX)_0$ is generated by the set $\set{x}{x \in
    X_0}$, and $\epsilon_0 \colon (QX)_0 \to X_0$ is specified by
    $\epsilon_0(x) = x$ and $d'_0 \colon (QX)_0 \to 0$ is the zero map;
\item For the inductive step, $(QX)_{i+1}$ (for $i \geqslant 0$) is
    generated by the set
\begin{equation*}
    \set{(x, z)}{x \in X_{i+1},\ z \in \ker d'_i,\ \epsilon_i(z) =
d_{i+1}(x)}\text,
\end{equation*}
whilst $\epsilon_{i+1} \colon (QX)_{i+1} \to X_{i+1}$ and $d'_{i+1} \colon
(QX)_{i+1} \to (QX)_i$ are specified by
\begin{equation*}
    \epsilon_{i+1}(x, z) = x \qquad \text{and} \qquad d'_{i+1}(x, z) =
    z\text.
\end{equation*}
\end{itemize}
Note that, in particular, we may view any $R$-module $M$ as a chain complex
concentrated in degree $0$; whereupon the above construction reduces to the
usual bar resolution of $M$. We can characterise a typical $\mathsf
Q$-coalgebra as being given by a chain complex $X$ equipped with, for each $i
\in \mathbb N$, a subset $G_i \subset X_i$ for which the inclusion map $G_i
\hookrightarrow X_i$ exhibits $X_i$ as the free $R$-module on $G_i$.
\end{Ex}

\subsection{Constructions on w.f.s.'s}
We end this section by reviewing two standard techniques for transferring
w.f.s.'s between categories that we shall need in the sequel. In both cases, we
assume the category $\C$ is locally presentable, so that we may freely apply
Proposition~\ref{refined}. For the first transfer technique, we consider
passage to the slice.

\begin{Prop}\label{slice}
  If $(\ELL, \R)$ is a weak factorisation system on $\C$, and \mbox{$X \in
  \C$}, then there is an induced weak factorisation system $(\ELL', \R')$
  on $\C / X$ for which $\ELL'$ and $\R'$ are the preimages of $\ELL$ and
  $\R$ under the forgetful functor $U \colon \C / X \to \C$. If $I$ is
  a set which cofibrantly generates $(\ELL, \R)$, then the set $I'$ of preimages of
  $I$ under $U$ generates $(\ELL', \R')$; and if we let $(\mathsf L,
  \mathsf R)$ and $(\mathsf L', \mathsf R')$ denote the universal
  algebraic realisations of $I$ and $I'$, then there is a functor
  $\tilde U \colon \mathsf R'\text-\cat{Alg} \to \mathsf
  R\text-\cat{Alg}$ making the following diagram a pullback:
\[
\cd{
  \mathsf R'\text-\cat{Alg} \ar[d] \ar[r]^{\tilde U} & \mathsf R\text-\cat{Alg} \ar[d] \\
  (\C / X)^\mathbf 2 \ar[r]_-{U^\mathbf 2} & \C^\mathbf 2\text.
}\]
\end{Prop}
\begin{proof}
  Mostly trivial; the final part follows from the characterisation of
  \mbox{$\mathsf R$-$\cat{Alg}$} given in Proposition~\ref{characteriseralg}.
\end{proof}
Our second transfer technique allows us to lift a cofibrantly generated weak
factorisation system along a right adjoint functor. This process was first
described in the general context of model categories by Sjoerd
Crans~\cite{Crans1995Quillen}.

\begin{Prop}\label{liftradj}
  Let $(\ELL, \R)$ be a cofibrantly generated w.f.s.\ on $\C$, and suppose that
  \mbox{$F \dashv G \colon \D \to \C$} with $\D$ locally presentable.  Then
  there is a w.f.s.\ $(\ELL', \R')$ on $\D$ for which $\R'$ is the
  preimage of $\R$ under $G$. Moreover, if $I$ is a generating set for
  $(\ELL, \R)$, then $I' = \set{Fi}{i \in I}$ is a generating set for
  $(\ELL', R')$; and if we let $(\mathsf L, \mathsf R)$ and $(\mathsf L',
  \mathsf R')$ denote the universal algebraic realisations of $I$ and
  $I'$, then there is a functor $\tilde G \colon \mathsf
  R'\text-\cat{Alg} \to \mathsf R\text-\cat{Alg}$ making the following
  diagram a pullback:
\[
\cd{
  \mathsf R'\text-\cat{Alg} \ar[d] \ar[r]^{\tilde G} & \mathsf R\text-\cat{Alg} \ar[d] \\
  \D^\mathbf 2  \ar[r]_{G^\mathbf 2} & \C^\mathbf 2\text.
}\]
\end{Prop}
\begin{proof}
The key observation is that there is a bijection between fillers for diagrams
of the forms
\begin{equation*}
\cd{
  U \ar[d]_i \ar[r]^f & GW \ar[d]^{Gp} \\
  V \ar[r]_g \ar@{.>}[ur]_{j} & GX
} \qquad \text{and} \qquad
\cd{
  FU \ar[d]_{Fi} \ar[r]^{\bar f} & W \ar[d]^{p} \\
  FV \ar[r]_{\bar g} \ar@{.>}[ur]_{\bar \jmath} & X\text,
}
\end{equation*}
where $\bar f$, $\bar g$ and $\bar \jmath$ denote the transposes of $f$, $g$
and $j$ under adjunction. The remaining details are straightforward.
\end{proof}

\section{Application to the theory of weak $\omega$-categories}
\subsection{The goal} By combining the material of the previous section with the techniques
outlined in the Introduction, we obtain a machinery that can weaken algebraic
structures in a canonical way. In this section, we will use this in the context
of Batanin's theory of weak $\omega$-categories~\cite{Batanin1998Monoidal} to
show that the canonical weakening of the theory of strict $\omega$-categories
is precisely the theory of weak $\omega$-categories singled out by Leinster
in~\cite{Leinster2004Operads}. Such a result is strongly hinted at in the work
of Batanin and Leinster (see in particular the remarks following Lemma 8.1
of~\cite{Batanin1998Monoidal}), but is never spelt out in detail; and so our
result serves as a clarification of the relationship between weak
$\omega$-categories and other kinds of weak algebraic structure.

\subsection{The ingredients} Recall that the key ingredients required for the
machinery of the Introduction are a base category $\C$; a notion of
higher-dimensionality on $\C$ arising from a weak factorisation system; a
category $\cat{Th}(\C)$ of theories on $\C$; and a particular theory $T \in
\cat{Th}(\C)$ we wish to weaken. We now describe each of these four ingredients
for our example.

\subsubsection{The base category} Our base category will be the category $\cat{GSet}$ of
\emph{globular sets}. This is the category $[\mathbb G^\op, \cat{Set}]$ of
presheaves over the \emph{globe category} $\mathbb G$, which in turn may be
presented as the free category on the graph
\begin{equation*}
    \cd{
    0 \ar@<3pt>[r]^-{\sigma_0} \ar@<-3pt>[r]_-{\tau_0} & 1 \ar@<3pt>[r]^-{\sigma_1} \ar@<-3pt>[r]_-{\tau_1} & 2 \ar@<3pt>[r]^-{\sigma_2} \ar@<-3pt>[r]_-{\tau_2} & \dots}
\end{equation*}
subject to the \emph{coglobularity equations} $\sigma_{n+1} \sigma_n =
\tau_{n+1} \sigma_n$ and $\sigma_{n+1} \tau_n = \tau_{n+1} \tau_n$ for all $n$.
Thus a globular set $X \in \cat{GSet}$ is given by sets $X_n$ of
\emph{$n$-cells} together with source and target functions $s_n, t_n \colon
X_{n+1} \to X_n$ subject to the \emph{globularity equations}, which assert that
the source and target $n$-cells of any $(n+1)$-cell are parallel.

\subsubsection{The weak factorisation system} \label{higeher} Our notion of higher-dimensionality on $\cat{GSet}$ will be obtained by a
\emph{Reedy category} technique~\cite{Reedy1974Homotopy}. The definition of a
Reedy category is quite subtle---see~\cite[Chapter 15]{Hirschhorn2003Model} for
instance---but we will not need the full generality here. Rather, we consider
the simpler notion of a \emph{direct category}; this being a small category
$\A$ which admits an identity-reflecting functor $\mathrm{dim} \colon \A \to
\gamma$ for some ordinal $\gamma$. For such a category $\A$, the presheaf
category $\hat \A := [\A^\op, \cat{Set}]$ comes equipped with canonical notions
of \emph{generating cell} and \emph{boundary}. The generating cells are the
representable presheaves $y(a) := \A(\thg, a)$; whilst their boundaries
$\partial y(a)$ are given by the coend
\begin{equation*}
    \partial y(a) := \!\!\!\!\!\!\!\!\!\!\int\limits^{\substack{b \in \A\\ \dim(b) < \dim(a)}}\!\!\!\!\!\!\!\!\!\! \A(b, a) \cdot y(b)
\end{equation*}
in $\hat \A$. The universal property of the displayed coend together with the
Yoneda lemma induces a canonical map of presheaves $\iota(a) \colon
\partial y(a) \to y(a)$; and so we obtain a set of generating cofibrations $I := \set{\iota(a)}{a \in \A}$.
Since any presheaf category is locally finitely presentable, we may apply
Proposition~\ref{refined} to obtain an algebraically realised w.f.s.\ on $\hat
\A$ generated by the set $I$.

The category $\mathbb G$ is a direct category, with $\gamma = \omega$ and
$\mathrm{dim}$ the unique identity-on-objects functor $\mathbb G \to \omega$;
and so applying the theory of the previous paragraph yields the following set
of generating cofibrations in $\cat{GSet}$:
  
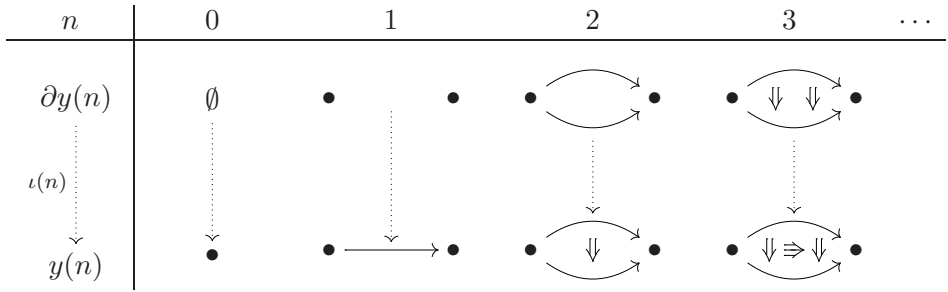
\[
  \begin{array}{c|*5{c}}
  n & \qquad 0 \qquad & \ 1 \ & \ 2 \ & \ 3 \  & \cdots \\ \hline \\
  \ccd[@R+2em@C+1em]{\partial y(n) \ar@{.>}[d]_{\iota(n)} \\ y(n)} &
  \ccd[@R+2em@C+1em]{\emptyset \ar@{.>}[d] \\ \bullet}&
  \ \ccd[@R+2em@C+1em]{ \bullet \ar@{}[r]_{}="a" & \bullet  \\ \bullet \ar[r]^{}="b" \ar@{.>}"a"; "b"& \bullet}\ &
  \ \ccd[@R+2em@C+1em]{ \bullet \ar@/^1em/[r] \ar@/_1em/[r]_{}="a" & \bullet  \\  \bullet \ar@/^1em/[r]^{}="b" \ar@/_1em/[r] \ar@{}[r] \ar@{=>}?(0.5)+/u  0.15cm/;?(0.5)+/d 0.15cm/ \ar@{.>}"a"; "b"& \bullet}\ &
  \ \ccd[@R+2em@C+1em]{ \bullet \ar@/^1em/[r] \ar@/_1em/[r]_{}="a" \ar@{}[r] \ar@{=>}?(0.65)+/u  0.15cm/;?(0.65)+/d 0.15cm/ \ar@{}[r] \ar@{=>}?(0.35)+/u  0.15cm/;?(0.35)+/d 0.15cm/ & \bullet  \\ \bullet \ar@{}[r] \ar@3?(0.5)+/l 0.13cm/;?(0.5)+/r 0.13cm/ \ar@/^1em/[r]^{}="b" \ar@/_1em/[r] \ar@{}[r] \ar@{=>}?(0.7)+/u  0.15cm/;?(0.7)+/d 0.15cm/ \ar@{}[r] \ar@{=>}?(0.3)+/u  0.15cm/;?(0.3)+/d 0.15cm/ & \bullet \ar@{.>}"a"; "b"}&
  \end{array}
  \]
\vskip0.5\baselineskip \noindent We may describe the  presheaves $\partial
y(n)$ explicitly as follows. We have that $\partial y(0) = 0$ and $\partial
y(1) = y(0) + y(0)$; whilst each subsequent $\partial y(n+2)$ is obtained as
the pushout
\begin{equation*}
\cd[@C+3em]{    y(n) + y(n) \ar[r]^{[y(\sigma_n),\, y(\tau_n)]} \ar[d]_{[y(\sigma_n),\, y(\tau_n)]} & y(n+1) \ar[d] \\ y(n+1) \ar[r] & \pullbackcorner  \partial y(n+2)\text.}
\end{equation*}
The inclusion maps $\iota(n) \colon \partial y(n) \to y(n)$ are given by taking
$\iota(0) \colon 0 \to y(0)$ to be the unique map; taking $\iota(1) \colon y(0)
+ y(0) \to y(1)$ to be $[y(\sigma_0),\, y(\tau_0)]$; and taking each subsequent
$\iota(n+2) \colon
\partial y(n+2) \to y(n+2)$ to be the map induced using the universal property of pushout with respect to the commutative square
\begin{equation*}
\cd[@C+3em]{    y(n) + y(n) \ar[r]^{[y(\sigma_n),\, y(\tau_n)]} \ar[d]_{[y(\sigma_n),\, y(\tau_n)]} & y(n+1) \ar[d]^{y(\sigma_{n+1})} \\ y(n+1) \ar[r]_{y(\tau_{n+1})} & y(n+2)\text.}
\end{equation*}

\subsubsection{The category of theories} We now give our notion of theory on $\cat{GSet}$. These will be
Batanin's \emph{globular operads}, which were introduced
in~\cite{Batanin1998Monoidal}; though our presentation of them will follow that
given in~\cite{Leinster2004Operads}. We may see globular operads as a
generalisation of $\cat{Set}$-based operads. Recall that we specify a
$\cat{Set}$-based operad by giving a collection $\set{\OH(n)}{n \in \mathbb N}$
of basic $n$-ary operations, together with data expressing how these operations
compose together. The collection of basic $n$-ary operations amounts to an
object $\OH$ of the category $\cat{Coll} = [\mathbb N, \cat{Set}]$: and any
such object induces a functor $\OH \otimes (\thg) \colon \cat{Set} \to
\cat{Set}$ given by
\begin{equation*}
    \OH \otimes X = \sum_{n} \OH(n) \times X^n\text.
\end{equation*}
In order for this functor to underlie a monad on $\cat{Set}$, we require that
$\OH$ should be a monoid with respect to the ``substitution'' tensor product on
$\cat{Coll}$, whose unit is $(0, 1, 0, 0, \dots)$, and whose binary tensor is
\begin{equation*}
    (A \otimes B)(n) = \sum_{\substack{k, n_1, \dots, n_k\\ n_1 + \dots + n_k = n}} A(n) \times B(n_k) \times \dots \times B(n_k)\text.
\end{equation*}
We call such a monoid an \emph{operad}; and define an algebra for an operad
$\OH$ to be an algebra for the induced monad $\OH \otimes (\thg)$ on
$\cat{Set}$.
\newcommand{\pd}{\mathrm{pd}}
We may specify globular operads and their algebras in a similar manner. First
we define the category $\cat{GColl}$ of collections of basic globular
operations. We can present this as the slice category $\cat{GSet} / \pd$, where
$\pd$ is the globular set of \emph{pasting diagrams}. If we write $(\thg)^\ast$
for the free monoid monad on $\cat{Set}$, then $\pd$ may be defined inductively
by
\begin{equation*}
    \pd_0 = \{\star\} \quad \text{and} \quad \pd_{n+1} = (\pd_n)^\ast\text,
\end{equation*}
with source and target maps given by $s_0 = t_0 = \mathord! \colon \pd_1 \to
\pd_0$, $s_{n+1} = (s_n)^\ast$ and $t_{n+1} = (t_n)^\ast$. However, we will
prefer to view $\cat{GColl}$ as $[\mathbb A^\op, \cat{Set}]$, where $\mathbb A$
is the category of elements of $\pd$. Any object $\OH \in [\mathbb A^\op,
\cat{Set}]$ induces a functor $\OH \otimes (\thg) \colon \cat{GSet} \to
\cat{GSet}$ given by
\begin{equation*}
    (\OH \otimes X)_i = \sum_{\pi \in \pd_i} \OH(\pi) \times \cat{GSet}(\hat \pi, X)\text,
\end{equation*}
where $\hat \pi$ is the realisation of $\pi$ as a globular set: see~\cite[\S
4.2]{Leinster2004Operads} for more details. In order for this functor to
underlie a monad on $\cat{GSet}$, we require that $\OH$ should be a monoid with
respect to the ``substitution'' tensor product on $\cat{GColl}$. A description
of this tensor product may be found in~\cite[\S 4.3]{Leinster2004Operads},
which we do not repeat since we do not need the details. We call a monoid with
respect to this tensor product a \emph{globular operad}; and define an algebra
for a globular operad $\OH$ to be an algebra for the induced monad $\OH \otimes
(\thg)$ on $\cat{GSet}$. A globular operad morphism is just a monoid morphism
in $\cat{GColl}$; and so we obtain a category $\cat{GOpd}$ of globular operads.

However, there is a small subtlety we must deal with. Part of the data for a
globular operad $\OH$ is a set of $0$-dimensional operations $\OH(\star)$,
where $\star$ is the unique element of $\pd_0$. The operad structure of $\OH$
descends to a monoid structure on the set $\OH(\star)$; and an $\OH$-algebra
structure on a globular set $X$ descends to a left action of $\OH(\star)$ on
the set of $0$-cells $X_0$. But if a globular operad $\OH$ is to represent a
theory of weak $\omega$-categories, then its monoid of $0$-dimensional
operations should be trivial, since we want the ``free weak $\omega$-category''
functor to be bijective on $0$-cells. In order for the general machinery to
take account of this fact, we take our category of theories to be the category
$\cat{NGOpd}$ of \emph{normalised globular operads}; this is the full
subcategory of $\cat{GOpd}$ whose objects are those globular operads with
$\OH(\star)$ a singleton. The restriction to normalised globular operads also
plays a central role in~\cite{Batanin2008Algebras}.

\subsubsection{The candidate theory} The fourth and final ingredient we require for our machinery is a theory $\T
\in \cat{NGOpd}$ which we wish to weaken. We take this to be the \emph{terminal
globular operad} $\T$ given by $\T(\pi) = 1$ for all $\pi \in \mathbb A$. This
embodies the theory of strict $\omega$-categories, in the sense that the
corresponding monad $\T \otimes (\thg)$ on $\cat{GSet}$ is the free strict
$\omega$-category monad.

\subsection{The transfer} \label{thetransfer} We now have all the ingredients needed for our machinery.
The first stage in applying it is to transfer the notion of
higher-dimensionality from $\cat{GSet}$ to $\cat{NGOpd}$. First we transfer
from $\cat{GSet}$ to $\cat{GColl} \cong \cat{GSet} / \pd$ using
Proposition~\ref{slice}. This yields a cofibrantly generated w.f.s. on
$\cat{GSet} / \pd$, with set of generating cofibrations
\begin{equation*}
    I':= \set{\cd[@-1em]{\partial y(n) \ar[rr]^-{\iota(n)} \ar[dr]_{\pi.\iota(n)} & & y(n) \ar[dl]^{\pi} \\ & \pd}}{n \in \mathbb N,\, \pi \in \pd_n}\text.
\end{equation*}
If we view $\cat{GColl}$ instead as $[\mathbb A^\op, \cat{Set}]$, then this set
 $I'$ has an alternative description. Indeed, $\mathbb A =
\mathrm{el}(\pd)$ is another example of a direct category so that the technique
described in \S\ref{higeher} may be applied; and it is easy to check that the
set $\set{\iota(\pi) \colon
\partial y(\pi) \to y(\pi)}{\pi \in \mathbb A}$ so obtained coincides with~$I'$.
The next step is to transfer this weak factorisation system from $\cat{GColl}$
to $\cat{NGOpd}$. We have adjunctions
\begin{equation}
    \cd[@C+1em]{
        \cat{NGOpd} \ar@<5pt>[r]^-{U} \ar@{}[r]|-{\top} &
        \cat{GOpd} \ar@<5pt>[r]^-{V} \ar@<5pt>[l]^-{F} \ar@{}[r]|-{\top} &
        \cat{GColl}\text: \ar@<5pt>[l]^-{H}
        }
\end{equation}
indeed, $\cat{NGOpd}$, $\cat{GOpd}$ and $\cat{GColl}$ are categories of models
for essentially-algebraic theories in the sense of Freyd; and both $U$ and $V$
are induced by forgetting essentially-algebraic structure, and so have left
adjoints. Essential algebraicity also implies that $\cat{NGOpd}$ is locally
finitely presentable, so that we may lift along $VU$ using
Proposition~\ref{liftradj} to obtain an algebraically realised w.f.s.\ on
$\cat{NGOpd}$, with set of generating cofibrations $I'' :=
\set{HF(\iota(\pi))}{\pi \in \mathbb A}$.

\subsection{The result} We are now ready to give our main result. We write $(\mathsf L, \mathsf R)$
for the universal algebraic realisation of the set of generating cofibrations
$I''$ in $\cat{NGOpd}$; and as in Remark \ref{cofibcomonad}, we write $\mathsf
Q$ for the cofibrant replacement comonad associated with $(\mathsf L, \mathsf
R)$.

\begin{Thm}\label{mainresult}
Applying the cofibrant replacement comonad $\mathsf Q$ of $\cat{NGOpd}$ to the
strict $\omega$-category operad $T$ yields the weak $\omega$-category operad
$L$ defined by Leinster in~\cite[\S 4]{Leinster2004Operads}.
\end{Thm}

In order to prove this, we must first recall what Leinster's operad $L$ is. The
central notion (cf.\ \cite[p.\ 139]{Leinster2004Operads}) is that of a
\emph{contraction} on an object of $C \in \cat{GColl}$. To give this we view
$C$ as a functor $\mathbb A^\op \to \cat{Set}$; now for each $\pi \in \pd_1$,
we define $P_\pi(C)$ to be the set $C(s_0(\pi)) \times C(t_0(\pi))$, whilst for
each $n \geqslant 2$ and $\pi \in \pd_{n}$, we define $P_\pi(C)$ to be the
pullback
\begin{equation*}
    \cd[@C+4em]{
        P_\pi(C) \pushoutcorner \ar[r] \ar[d] & C(s_{n-1}(\pi)) \ar[d]^{(s_{n-2}, t_{n-2})} \\
        C(t_{n-1}(\pi)) \ar[r]_-{(s_{n-2}, t_{n-2})} & C(s_{n-2}s_{n-1}(\pi)) \times C(t_{n-2}s_{n-1}(\pi))\text.
    }
\end{equation*}
A \emph{contraction} $\kappa$ on $C$ is now given by functions $\kappa_\pi
\colon P_\pi(C) \to C(\pi)$ for each $n \geqslant 1$ and $\pi \in \pd_n$ which
render commutative the evident triangles\label{pagereftriangles}
\begin{equation*}
    \cd[@C-3em]{P_\pi(C) \ar[dr] \ar[rr]^{\kappa_\pi} & & C(\pi) \ar[dl] \\ & C(t_{n-1}(\pi)) \times C(s_{n-1}(\pi))\text.}
\end{equation*}
Any morphism $f \colon C \to D$ in $\cat{GColl}$ induces morphisms $P_\pi(f)
\colon P_\pi(C) \to P_\pi(D)$ for each $n \geqslant 1$ and $\pi \in \pd_n$, so
that if $\kappa$ and $\lambda$ are contractions on $C$ and $D$ respectively, we
may say that $f$ \emph{preserves the contraction} just when $f(\pi) .
\kappa_\pi = \lambda_\pi . P_\pi(f)$. We now define the category $\cat{OWC}$ of
\emph{operads with contraction} to have:
\begin{itemize}
\item \textbf{Objects} being pairs $(\K, \kappa)$, where $\K \in
    \cat{GOpd}$ and $\kappa$ is a contraction on $U(K)$;
\item \textbf{Morphisms} $f \colon (\K, \kappa) \to (\K', \kappa')$ being
    maps $f \colon \K \to \K'$ of globular operads for which $U(f)$
    preserves the contraction.
\end{itemize}
The operad $L$ of Theorem \ref{mainresult} is now defined to be the underlying
operad $L$ of the initial object $(L, \lambda)$ of $\cat{OWC}$.

\begin{proof}
First note that as well as a cofibrant replacement comonad $\mathsf Q$, we also
have an ``acyclically fibrant replacement monad'' $\mathsf P$ on $\cat{NGOpd}$,
obtained by restricting and corestricting $\mathsf R$ to the slice over $\T$.
The object $\mathsf Q(\T)$ that we are interested in is given by the
universally determined factorisation of the unique map $\I \to \T$, where $\I$
is the initial object of $\cat{NGOpd}$. But this is equally well a description
of $\mathsf P(\I)$. It follows that we may characterise $\mathsf Q(\T)$ as the
underlying normalised operad of the initial $\mathsf P$-algebra.

Let us now use Propositions~\ref{characteriseralg}, \ref{slice} and
\ref{liftradj} to give an explicit description of the category of $\mathsf
P$-algebras. Recall from \S \ref{higeher} the set of maps $I = \set{\partial
y(n) \to y(n)}{n \in \N}$ in $\cat{GSet}$. Let us write $(\mathsf L_\cat{GSet},
\mathsf R_\cat{GSet})$ for the corresponding universally determined algebraic
realisation, and $\mathsf P_{\cat{GSet}/ \pd}$ for the restriction and
corestriction of $\mathsf R_\cat{GSet}$ to the slice over $\pd \in \cat{GSet}$.
Propositions~\ref{slice} and~\ref{liftradj} now tell us that we have a pullback
diagram
\begin{equation*}
    \cd{
      \mathsf P\text-\cat{Alg} \ar[r] \ar[d] & \mathsf P_{\cat{GSet}/ \pd}\text-\cat{Alg} \ar[d] \\
      \cat{NGOpd} \ar[r]_{VU} & \cat{GSet / \pd}\text;
    }
\end{equation*}
whilst Proposition~\ref{characteriseralg} provides us with an explicit
description of the category $\mathsf P_\cat{GSet / \pd}\text-\cat{Alg}$.
Putting these facts together, we find that to give an object of $\mathsf
P$-$\cat{Alg}$ is to give a normalised globular operad $\C$, together with
chosen fillers for every diagram of the form
\begin{equation}\label{diagfillers}
    \cd{
        \partial y(n) \ar[d]_{\iota(n)} \ar[r] & C \ar[d]^c \\
        y(n) \ar[r] \ar@{.>}[ur] & \pd\text,
    }
\end{equation}
where the arrow down the right is the underlying globular collection of $\C$.
Using the explicit construction of the maps $\iota(n)$ given in \S
\ref{higeher}, we see that to give chosen fillers in \eqref{diagfillers} is
trivial when $n = 0$ (by normalisation of $\C$); and that for $n \geqslant 1$,
it is precisely to give the functions $\kappa_\pi \colon P_\pi(C) \to C(\pi)$
described following the statement of Theorem~\ref{mainresult}, and so amounts
to giving a contraction on $C$. A similar argument shows that to give a
morphism of $\mathsf P$-$\cat{Alg}$ is to give a morphism of underlying
normalised globular operads $\C \to \D$ which preserves the contraction. Thus
we obtain a pullback diagram
\begin{equation}\label{pbsquare}
    \cd{
        \mathsf P\text-\cat{Alg} \ar[r]^{\tilde U} \ar[d] & \cat{OWC} \ar[d] \\
        \cat{NGOpd} \ar[r]_U & \cat{GOpd}\text.
    }
\end{equation}
By the remarks at the start of the proof, we will be done if we can show that
$\tilde U$ sends the initial object of $\mathsf P$-$\cat{Alg}$ to the initial
object of $\cat{OWC}$. Now, as we noted in \S 3.3, the functor $U \colon
\cat{NGOpd} \to \cat{GOpd}$ has a left adjoint; and an application of the
adjoint lifting theorem~\cite[Theorem 2]{Johnstone1975Adjoint} shows that
$\tilde U$ also has a left adjoint. Note that here we need the fact that
\mbox{$\mathsf P$-$\cat{Alg}$} is again describable in terms of
essentially-algebraic structure, and so cocomplete. Furthermore, $\tilde U$ is
fully faithful, because $U$ is and \eqref{pbsquare} is a pullback; and so we
may identify $\mathsf P$-$\cat{Alg}$ with a reflective subcategory of
$\cat{OWC}$. Thus we will be done if we can show that the initial object $(L,
\lambda)$ of $\cat{OWC}$ lies in this reflective subcategory; in other words,
if we can show that $L$ is normalised. But this is known to be the case:
see~\cite{Cheng2003Monad}, for example.
\end{proof}

\begin{Rk}
Note that the restriction to normalised globular operads is vital for the above
proof to go through. Indeed, were we to take the universal cofibrant
replacement for $\T$ in the category $\cat{GOpd}$ rather than $\cat{NGOpd}$,
then we would no longer obtain Leinster's operad $L$. By following the steps of
the above proof, we find that what we obtain instead is the initial
operad-with-augmented-contraction, where an \emph{augmented contraction} on a
collection $C \in \cat{GColl}$ is given by a contraction on $C$ together with a
chosen element of the set $C(\star)$ of $0$-dimensional operations. The initial
operad-with-augmented-contraction is no longer normalised; and in fact, it is
not hard to show that its monoid of $0$-dimensional operations is the free
monoid on one generator.
\end{Rk}

%%%% Bibliography
\bibliographystyle{acm}
\bibliography{rhgg2}

\end{document}